\documentclass{amsart}

\usepackage{tikz}
\usetikzlibrary{matrix,decorations.pathreplacing,calc}

\usepackage{enumitem}

\tikzset{ 1
table/.style={
  matrix of math nodes,
  row sep=-\pgflinewidth,
  column sep=-\pgflinewidth,
  nodes={rectangle,text width=3em,align=center},
  text depth=1.25ex,
  text height=2.5ex,
  nodes in empty cells,
  left delimiter=[,
  right delimiter={]},
  ampersand replacement=\&
}
}

\usepackage{amscd}
\usepackage{bbm}
\usepackage{amssymb}
\usepackage{amsmath}
\usepackage{amsfonts}
\usepackage{amsthm}
\usepackage{stmaryrd}
\usepackage[all]{xy}
\usepackage{mathrsfs}
\usepackage{graphicx}
\usepackage{hyperref}
\usepackage{color}
\usepackage{tikz}
\usetikzlibrary{matrix}
\usepackage{amsmath,amscd}
\usepackage{pifont}

\usetikzlibrary{arrows}

\numberwithin{equation}{subsection}
\newtheorem{theorem}[subsection]{Theorem}

\newtheorem{corollary}[subsection]{Corollary}
\newtheorem{lemma}[subsection]{Lemma}
\newtheorem{proposition}[subsection]{Proposition}

\theoremstyle{definition}

\newtheorem{definition}[subsection]{Definition}

\newtheorem{notation}[subsection]{Notation}

\def\calC{\mathcal{C}}

\def\calO{\mathcal{O}}

\def\CC{\mathbb{C}}
\def\FF{\mathbb{F}}
\def\GG{\mathbb{G}}

\def\QQ{\mathbb{Q}}

\def\SS{\mathbb{S}}

\def\ZZ{\mathbb{Z}}

\DeclareMathOperator{\NP}{NP}

\newcommand{\Diag}{\mathrm{Diag}}

\newcommand{\Norm}{\mathrm{Norm}}

\newcommand{\HP}{\mathrm{HP}}
\newcommand{\UP}{\mathrm{UP}}

\newcommand{\Tr}{\mathrm{Tr}}

\pgfkeys{tikz/mymatrixenv/.style={decoration=brace,every left delimiter/.style={xshift=3pt},every right delimiter/.style={xshift=-3pt}}}
\pgfkeys{tikz/mymatrix/.style={matrix of math nodes,left delimiter=[,right delimiter={]},inner sep=2pt,column sep=1em,row sep=0.5em,nodes={inner sep=0pt}}}
\pgfkeys{tikz/mymatrixbrace/.style={decorate,thick}}

\begin{document}\large

\title{Newton slopes for twisted Artin--Schreier--Witt Towers} 

\author{Rufei Ren}

\address{University of Rochester, Department of
	Mathematics,  Hylan Building, 140 Trustee Road, Rochester, NY 14627}

\email{rren2@ur.rochester.edu}
\date{\today}

\begin{abstract}
We fix a monic polynomial $f(x) \in \FF_q[x]$ over a finite field of characteristic $p$ of degree relatively prime to $p$. Let $a\mapsto \omega(a)$ be the Teichm\"uller lift of $\FF_q$, and let $\chi:\ZZ_p\to \CC_p^\times$ be a finite character of $\ZZ_p$. 
The $L$-function associated to the polynomial $f$ and the so-called twisted character $\omega^u\times \chi$ is denoted by $L_f(\omega^u,\chi,s)$ (see Definition~\ref{L-function}).  
We prove that, when the conductor of the
character is large enough, the $p$-adic Newton slopes of this $L$-function form arithmetic progressions.
\end{abstract}

\subjclass[2010]{11T23 (primary), 11L07 11F33 13F35 (secondary).}
\keywords{Artin--Schreier--Witt towers, $T$-adic exponential sums, Slopes of Newton polygon, $T$-adic Newton polygon for Artin--Schreier--Witt towers, Eigencurves}
\maketitle

\setcounter{tocdepth}{1}
\tableofcontents

\section{Introduction}
Let $p$ be a prime number. Let $\FF_q$ be the finite field of $q=p^a$ elements. 
Let \begin{align*}
\omega:&\FF_q\to\ZZ_q\\
&\alpha\ \mapsto \omega(\alpha):=\hat \alpha
\end{align*} be the Teichm\"uller lift of $\FF_q$. For any $u\in\{0,1,\cdots, q-1\}$, we put \begin{align*}
\omega^u: &\FF_q\to\ZZ_q\\
&\alpha\ \mapsto \omega^u(\alpha)=\hat \alpha^u.
\end{align*} 
We fix a monic polynomial $f(x) = x^d +a_{d-1}x^{d-1} + \cdots + a_0 \in \FF_q[x]$ of degree $d$ which is coprime to $p$. 
Set $ a_d=1$ and put $\hat a_i: = \omega( a_i)$ for $i=0, \dots, d$. The \emph{Teichm\"uller lift} of the polynomial $f(x)$ is defined by 
$\hat f(x):=x^d + \hat a_{d-1}x^{d-1} + \cdots + \hat a_0 \in \ZZ_q[x].$ 

The non-twisted $L$-function associated to $f(x)$ and a finite character $\chi:\ZZ_p\to \CC_p^\times$ is defined as 
	\begin{equation}
L_f(\chi,s)= \prod\limits_{x \in |\GG_{m,\FF_q}|}
\frac{1}{1-\chi\Big(\Tr_{\QQ_{q^{\deg(x)}}/\QQ_{p}}\big(\hat f(\hat x)\big)\Big)s^{\deg(x)}},
\end{equation}
where $\GG_{m,\FF_q}$ is the one-dimensional torus over $\FF_q$ and $\deg(x)$ stands for the degree of $x$.

In \cite{Davis-wan-xiao}, Davis, Wan, and Xiao  proved that 
\begin{itemize}
	\item If $\chi_1$ and $\chi_2$ are two finite characters with the same conductor $p^m\geq \frac{ap(d-1)^2}{8d}$, then $L_f(\chi_1,s)$ and $L_f(\chi_2,s)$ have the same Newton polygon.
	\item Let $\chi_0:\ZZ_p\to \CC_p^\times$ be a fixed character with conductor $p^{\lceil\log_p \frac{ap(d-1)^2}{8d}\rceil}.$ Then the $p$-adic Newton slopes of $L_f(\chi_1,s)$ (see Definition~\ref{Newton slopes}) form a disjoint union of arithmetic progressions determined by the $p$-adic Newton slopes of $L_f(\chi_0,s)$.
\end{itemize}
In \cite{bfz}, Blache, Ferard, and Zhu studied the so-called twisted $L$-functions (see Definition~\ref{L-function}), whose $p$-adic Newton polygons satisfy a universal lower bound proved by C.Liu and W.Liu in \cite{liu-liu}. This lower bound is similar to the one given in \cite{Davis-wan-xiao}. Therefore, 
it is of interest to ask  if the $p$-adic Newton slopes of the twisted $L$-functions also form arithmetic progressions. In this paper, we give an upper bound for the twisted $L$-function and prove that it coincides with the lower bound at $x=kd$ for any integer $k\geq0$. As a consequence, we prove that its $p$-adic Newton slopes indeed form arithmetic progressions.

\begin{notation}
For an integer $u$ in the set $\{0,1,\dots,q-2\}$, we put $$u=\sum\limits_{i=0}^{a-1} u(i)p^i,$$ where $0\leq u(i)\leq p-1$ for any $0\leq i\leq a-1$. 
\end{notation}

\begin{definition}\label{L-function}
	Let $\chi:\ZZ_p\to \CC_p^\times$ be a finite character with conductor $p^{m_\chi}$. The \emph{twisted $L$-function} associated to the characters $\chi$ and $\omega^u$ is defined by 
	\begin{equation}
	L_f(\omega^u,\chi,s)= \prod\limits_{x \in |\GG_{m,\FF_q}|}
	\frac{1}{1-\omega^u\circ\Norm_{\FF_{q^{\deg(x)}}/\FF_q} (x)\cdot\chi\Big(\Tr_{\QQ_{q^{\deg(x)}}/\QQ_{p}}\big(\hat f(\hat x)\big)\Big)s^{\deg(x)}},
	\end{equation}
	where $\GG_{m,\FF_q}$ is the one-dimensional torus over $\FF_q$ and $\deg(x)$ stands for the degree of $x$.
	In \cite{liu-wei}, Liu and Wei prove that the $L$-function $L_f(\omega^u,\chi,s)$ is a polynomial of degree $dp^{m_\chi-1}$.
\end{definition}

\begin{notation}
	For simplicity of notations, we denote $$y_u(k):=\frac{ak(k-1)(p-1)}{2d}+\frac k d\sum\limits_{i=0}^{a-1}u(i).$$
\end{notation}
\begin{definition}\label{Newton slopes}
	We call the slopes of the line segments of the $p$-adic Newton polygon of $L_f(\omega^u, \chi, s)$ the \emph{$p$-adic Newton slopes} of $L_f(\omega^u, \chi, s)$.
\end{definition}
In this paper, we prove the following.
\begin{theorem}\label{main}
	\hfill
	\begin{itemize}
		\item[(a)]
		The $p$-adic Newton polygon of $L_f(\omega^u,\chi, s)$ passes through the points $$\big(kd, \frac{y_u(kd)}{(p-1)p^{m_\chi-1}}\big)\quad \textrm{for any}\ 0\leq k\leq p^{m_\chi-1}.$$
		\item[(b)] The $p$-adic Newton polygon of $L_f(\omega^u, \chi, s)$ has slopes (in increasing order) 
		\[\bigcup_{k=1}^{p^{m_\chi-1}} \{\alpha_{k 1},\alpha_{k 2},\dots,\alpha_{k d}\}, 
		\]
		where
		\[\frac{a(k-1)}{p^{m_{\chi}-1}}+\frac{\sum\limits_{i=0}^{a-1}u(i)}{d(p-1)p^{m_\chi-1}}\leq \alpha_{k j}\leq \frac{a(k-1)}{p^{m_{\chi}-1}}+\frac{\sum\limits_{i=0}^{a-1}u(i)}{d(p-1)p^{m_\chi-1}}+
		\frac{a(d-1)}{dp^{m_\chi-1}}\]
	\end{itemize}
for any $1\leq j\leq d$.
\end{theorem}

When the conductor $p^{m_\chi}$ of $L_f(\omega^u,\chi, s)$ is large enough, the $p$-adic Newton slopes of $L_f(\omega^u,\chi, s)$ have the following property.
\begin{theorem}[Main theorem]\label{strong}
	Let $m_0$ be the minimal positive integer such that $p^{m_0}> \frac{adp}{8(p-1)}$
	and let $0\leq \alpha_1,\alpha_2,\dots,\alpha_{dp^{m_0-1}}<a$
	denote the slopes of the $p$-adic Newton polygon of $L_f(\omega^u,\chi_0,s)$ for a
	finite character $\chi_0:\ZZ_p\to \CC_p^\times$
	with $m_{\chi_0} = m_0$. Then for any finite character $\chi:\ZZ_p\to \CC_p^\times$
	with $m_\chi\geq m_0$, the $p$-adic Newton polygon of $L_f(\omega^u,\chi,s)$ has slopes
	\begin{equation}
	\bigcup_{i=0}^{p^{m_\chi-m_0}-1}\Big\{\frac{\alpha_1+ai}{p^{m_\chi-m_0}},\frac{\alpha_2+ai}{p^{m_\chi-m_0}},\dots, \frac{\alpha_{dp^{m_0-1}}+ai}{p^{m_\chi-m_0}}\Big\}.
	\end{equation}
\end{theorem}
Theorem~\ref{strong} says that when $m_\chi$ is large enough, the $p$-adic Newton slopes of $L_f(\omega^u,\chi, s)$ form a disjoint union of arithmetic progressions determined by the $p$-adic Newton slopes of $L_f(\omega^u,\chi_0, s)$. A similar result is proved by Li in \cite{li} for the general Witt towers without twisting.

This paper is inspired by the $p$-adic Newton slopes of $L_f(\chi,s)$ in arithmetic progressions (proved in \cite{Davis-wan-xiao}), the twisted decomposition of $$L_{f(x^{q-1})}(\chi, s):=\prod_{u=0}^{q-2} L_f(\omega^u,\chi,s)$$ in 
\cite{bfz}, and the lower bound for the Newton polygon of $L_{f}(\omega^u,\chi, s)$ given in \cite{liu-liu}. Let $$\calC_0\leftarrow \calC_1\leftarrow\cdots\leftarrow \calC_m\leftarrow\cdots$$ be the Artin--Schreier--Witt curve tower associated to the polynomial $f(x^{q-1})$, and let $Z(\calC_m, s)$ be the zeta function of the curve $\calC_m$. 
It is known that
 \[\begin{cases}
L_{f}(\omega^u,\chi, s)&\textrm{if}\ u\neq 0\\
\frac{L_{f}(\omega^u,\chi, s)}{1-\chi(\hat f(0))}&\textrm{if}\ u= 0
\end{cases}\] are factors of $Z(\calC_m, s)$, and the degree of $L_f(\omega^u,\chi_0,s)$ is $\frac{1}{q-1}$ of the degree of $L_f(\chi_0,s))$. Therefore, as a corollary of Theorem~\ref{strong}, we give a more precise description of zeros of $Z(\calC_m, s)$ than the one given in \cite{Davis-wan-xiao}. 
After we posted this paper on Arxiv, we are informed that there is a similar result obtained independently by Liu, Liu, and Niu.
\subsection*{Acknowledgments}
The author thanks Dennis A. Eichhorn, Karl Rubin, Daqing Wan and Liang Xiao for many valuable discussions and suggestions.
	
\section{Notation}
In this section, we introduce some notations that we will use through out the paper.
\begin{notation}
	 We write $v_{T}(-)$ for the $T$-adic valuation of elements in $\CC_p[T]$ and 
	$v_{p}(-)$ for the $p$-adic valuation of elements in $\CC_p$.
\end{notation}

\begin{definition}\label{Newton polygon}
	Given a set $S:=\big\{(k,d_k)\;\big|\;0\leq k\leq n\big\}$. The \emph{Newton polygon} of $S$, denoted by $\NP(S)$, is the lower convex hull of points in $S$. We call $n$ \emph{the length} of $\NP(S)$.
	
	For a power series $F(s)=\sum\limits_{i=0}^n \mathbbm{u}_i(T)s^i$, we put $$\NP_T(F):=\NP\Big(\Big\{(k,v_T(\mathbbm u_k))\;\big|\;0\leq k\leq n\Big\}\Big),$$
	where $n\in \ZZ_{\geq 0}\cup\infty$.
\end{definition}

\begin{definition}\label{definition of R}
For a Newton polygon $\NP$, we write $R(\NP)$ for
the multiset of slopes in $\NP$.

It has an inverse, denoted by $R^{-1}$, mapping a multiset $\SS$ to the lower convex whose slopes coincide with this multiset. 
\end{definition}

\begin{notation}\label{properties for NP}
	\hfill
	\begin{itemize}
		\item[(a)] Let $\SS_1$ and $\SS_2$ be two multisets in $\QQ$. We denote by $$\SS_1\uplus\SS_2$$ the union of $\SS_1$ and $\SS_2$ as multisets.
	\item[(b)] For any two Newton polygons $\NP_1$ and $\NP_2$, we write $$\NP_1\oplus\NP_2$$
	for the Newton polygon whose slopes are the union of the slopes of $\NP_1$ and $\NP_2$.
	\item[(c)] We denote by
	$\NP(i)$ the height of  $\NP$ at $x=i$. 
	\item[(d)]For any $t\in \QQ$, we denote $t+\NP$
	be the Newton polygon such that $$(t+\NP)(k)=\NP(k)+t\quad\textrm{for any}\ 0\leq k\leq n,$$
	where $n$ is the length of $\NP$.
	\end{itemize}
\end{notation}
\begin{definition}
Let  $\NP_1$ and $\NP_2$ be two polygons of same length $n$. If $$\NP_1(k)\geq NP_2(k)$$ holds for any $0\leq k\leq n$, then we call that $\NP_1$ is above $\NP_2$ and denote this by $\NP_1\geq \NP_2$.
\end{definition}
\begin{lemma}\label{sum inequality}
	If $\Big\{\NP_{1,i}\;\big|\;1\leq i\leq m\Big\}$ and  $\Big\{\NP_{2,i}\;\big|\;1\leq i\leq m\Big\}$ are two sets of Newton polygons such that for any $1\leq i\leq m$, \begin{itemize}
		\item $\NP_{1,i}$ and $\NP_{2,i}$ have the same length, and
		\item $\NP_{1,i}\geq\NP_{2,i}$,
	\end{itemize}then \[\bigoplus\limits_{i=1}^m \NP_{1,i}\geq \bigoplus\limits_{i=1}^m \NP_{2,i}.\]
\end{lemma}
\begin{proof}
It follows directly from the definition of ``$\oplus$''. 
\end{proof}
\begin{definition}
	For any positive integer $\ell$, the sum \[S_{f,\omega^u,\chi}(\ell,s)=\sum_{x\in\FF^\times_{q^\ell}}\omega^u\circ\Norm_{\FF_{q^\ell}/\FF_q} (x)\cdot(1+T)^{\Big(\Tr_{\QQ_{q^{\ell}}/\QQ_{p}}\big( \hat f(\hat x)\big)\Big)}\in \ZZ_q[\![T]\!]\] is called a twisted $T$-adic exponential sum of $f(x)$. 
\end{definition}
	The $L$-function $L_f(\omega^u,T,s)$ (see Definition~\ref{L-function}) satisfies
\[L_f(\omega^u,T,s)= \exp(\sum_{\ell=1}^{\infty}S_{f,\omega^u,\chi}(\ell,s)\frac{s^\ell}{\ell}).\]
It is easy to check 
\begin{equation}\label{specialization of L}
L_f(\omega^u,T,s)\big|_{T=\chi(1)-1}=L_f(\omega^u,\chi,s).
\end{equation}

\begin{lemma}\label{core lemma}
	If we put $g(x):=f(x^{q-1})$, then
	\begin{equation}\label{core eq}
	L_g(\omega^0,T,s)=\prod_{u=0}^{q-2} L_f(\omega^u,T,s).
	\end{equation}
\end{lemma}
\begin{proof}
	Put $$(1+T)^{\Big(\Tr_{\QQ_{q^{\ell}}/\QQ_{p}}\big(\hat f(\hat x)\big)\Big)}:=\sum_{n=0}^\infty b_{\ell, n}(T)\hat x^n\in \ZZ_p[\![T]\!][\![\hat x]\!],$$
	where $b_{\ell, n}(T)\in \ZZ_p[\![T]\!]$ for all $n\geq 0$ and $\ell\geq 1$.
	
	Notice that for each $u\in\{0,1,\dots,q-2\}$, we have  $$S_{f,\omega^u,\chi}(\ell,s)=(q^\ell-1)\sum_{n=1}^\infty b_{\ell, (q^\ell-1)(n-\frac{u}{q-1})}(T).$$ 
	Therefore, by taking the sum of $u$ over the set $\{0,1,\dots, q-2\}$, we get 
\[
	\sum_{u=0}^{q-2}S_{f,\omega^u,T}(\ell,s)=(q^\ell-1)\sum_{u=0}^{q-2}\sum_{n=1}^\infty b_{\ell, (q^\ell-1)(n-\frac{u}{q-1})}(T)=(q^\ell-1)\sum_{n=1}^\infty b_{\ell, \frac{n(q^\ell-1)}{q-1}}(T).
\]
	
	On the other hand, by definition, it is easy to check that  $$S_{g,\omega^0,T}(\ell,s)=(q^\ell-1)\sum_{n=1}^\infty b_{\ell, \frac{n(q^\ell-1)}{q-1}}(T).$$Therefore, we have 
	\[S_{g,\omega^0,T}(\ell,s)=\sum_{u=0}^{q-2}S_{f,\omega^u,T}(\ell,s),\]
	for all $\ell\geq 1$, which implies 
	\[L_g(\omega^0,T,s)=\prod_{u=0}^{q-2} L_f(\omega^u,T,s).\qedhere\]
\end{proof}
\begin{definition}
	The \emph{characteristic power series} of $f$ is given by
	\begin{equation}
	C_f(\omega^u,T,s):=\prod_{i=0}^{\infty}L_f(\omega^u,T,q^is),
	\end{equation}
	which is shown as a $p$-adic entire power series in \cite{liu-liu-niu}.
\end{definition}
By Lemma~\ref{core lemma}, we know that 
$$C_g(\omega^0,T,s)=\prod_{u=0}^{q-2} C_f(\omega^u,T,s).$$

\begin{notation}
We denote by $\NP_f(L,\omega^u,T)$ (resp. $\NP_f(L,\omega^u,\chi)$) the $T$-adic Newton polygon (resp. $p$-adic Newton polygon) of $L_f(\omega^u,T,s)$ (resp. $L_f(\omega^u,\chi,s)$). 

Similarly, we write $\NP_f(C,\omega^u,T)$ and $\NP_f(C,\omega^u,\chi)$ for the $T$-adic Newton polygon (resp. $p$-adic Newton polygon) of $C_f(\omega^u,T,s)$ and $C_f(\omega^u,\chi,s)$ respectively.	
\end{notation}

\section{The T-adic Dwork's Trace Formula}
In this section, we recall properties of the $L$-function associated to a $T$-adic exponential
sum as considered by Liu and Wan in \cite{liu-wan}. Its specializations to
appropriate values of $T$ interpolate the $L$-functions considered above.

\begin{notation}
	We first recall that the \emph{Artin--Hasse exponential series} is defined by
	\begin{equation}\label{Artin-Hasse}
	E(\pi) = \exp\big( \sum_{i=0}^\infty \frac{\pi^{p^i}}{p^i} \big) = \prod\limits_{p \nmid i,\ i \geq 1} \big( 1-\pi^i\big)^{-\mu(i)/i} \in 1+ \pi + \pi^2 \ZZ_p[\![ \pi ]\!].
	\end{equation}
	Setting $T = E(\pi) -1$ defines an isomorphism $\ZZ_p\llbracket \pi \rrbracket \cong \ZZ_p\llbracket T\rrbracket$.
\end{notation}

\begin{notation}
	For our given polynomial $f(x) = \sum\limits_{i=0}^d a_i x^i \in \ZZ_q[x]$, we put
	\begin{equation}
	\label{E:Ef(x)}
	E_f(x) := \prod\limits_{i=0}^d E(a_i \pi x^i) \in \ZZ_q[\![\pi]\!] [\![ x ]\!].\end{equation}
\end{notation}
We follow the notation of \cite{liu-liu}. Set $$C_u:=\{v\in \ZZ_{\geq 0}\;\big|\; v\equiv u\pmod{q-1}\}$$ and $$\textbf{B}_u:=\Big\{\sum_{v\in C_u} b_vT^{\frac{v}{d(q-1)}}x^{\frac{v}{(q-1)}}\;\big|\; b_v\in \ZZ_q[\![T^{\frac{v}{d(q-1)}}]\!]\ \textrm{and}\ v_T(b_v)\to \infty \Big\}.$$

\begin{notation}
	\hfill
	\begin{itemize}
	\item[(a)] For two integers $n$ and $m$, we denote by $n\%m$ the residue class of $n$ modulo $m$ in $\{0,1,\dots,m-1\}$. 
\item[(b)] Recall $u\in \{0,1,\dots,q-2\}$. We write $b_u|a$ for the minimal positive integer such that $u^{p^{b_u}}\equiv u \pmod q$. 
\item[(c)] Denote $$u_i:=(u^{p^i})\% (q-1)\ \textrm{for}\ i=0,\dots, b_u-1$$ and put \[\widetilde{\textbf{B}}_u=\bigoplus_{i=0}^{b_u-1} \textbf{B}_{u_i}\] to be the total Banach space associated to $u$.
\item[(d)] Choose a permutation $(i_1,i_2,\dots,i_{b_u})$ of $\{1,2,\dots,b_u\}$ such that the sequence $\{u_{i_n}\}$ is non-decreasing. Put   $$\biguplus\limits_{i=0}^{b_{u}-1} C_{u_i}:=\big(c_{u,n}\big)_{n\in \ZZ_{\geq 0}}$$ to be a non-decreasing sequence. 	
\end{itemize}

It is easy to check that 
	\begin{equation}\label{c}
	c_{u,n}=(q-1)\left\lfloor\frac{n}{b_u}\right\rfloor+u_{i_{(n\%b_u)}}.
	\end{equation}
\end{notation}
Let $\psi_p$ denote the operator on $\widetilde{\textbf{B}}_u$ given by 
$$\psi_p\Big(\sum_{n\geq 0}^\infty d_n(T) x^n\Big): = \sum_{n\geq 0}^\infty d_{pn}( T) x^n,$$ and let $\psi$  be the composite linear operator 
\begin{equation}
\label{E:psi}
\psi := \sigma\circ\psi_p \circ E_f(x): \widetilde{\textbf{B}}_u \longrightarrow \widetilde{\textbf{B}}_u,
\end{equation}
where $\sigma$ is the Frobenius automorphism of $\ZZ_q$, and $E_f(x)$ acts on $g\in \widetilde{\textbf{B}}_u$ by 
$$E_f(x)(g):=E_f(x) \cdot g. $$

By Dwork's trace formula, we have 
\begin{lemma}
	The characteristic power series $C_f(\omega^u,T,s)$ satisfies 
	\begin{equation}
	C_f(\omega^u,T,s)=\det\Big(1-\psi^as\;\big|\;\textbf{B}_u/\ZZ_q[\![T^{\frac{1}{d(q-1)}}]\!]\Big).
	\end{equation}
\end{lemma}
\begin{proof}
	See \cite[Theorem~2.1]{liu-liu}.
\end{proof}
By \cite[Lemma~4.2]{liu-liu}, we have
\begin{equation}\label{ll}
C_f(\omega^u,T,s)^{b_u}=\det\Big(1-\psi^as\;\big|\;\widetilde{\textbf{B}}_u/\ZZ_q[\![T^{\frac{1}{d(q-1)}}]\!]\Big).
\end{equation}

We write $$B=\big(T^{\frac{v}{d(q-1)}}x^{\frac{c_{u,n}}{(q-1)}}\big)_{n\in \ZZ_{\geq 0}}$$ for a basis of $\widetilde{\textbf{B}}_u$ over $\ZZ_q[\![T^{\frac{1}{d(q-1)}}]\!]$ and denote by $N$ the standard matrix of $\psi$ associated to the basis $B$.

It is not hard to check that $N$ is an infinite dimensional matrix of the form 
\begin{equation}
\label{E:explicit N}
N=\begin{pmatrix} 
*T^{\frac{(p-1)c_{u,0}}{d(q-1)}}&*T^{\frac{(p-1)c_{u,0}}{d(q-1)}}&\cdots& *T^{\frac{(p-1)c_{u,0}}{d(q-1)}}&\cdots\\
*T^{\frac{(p-1)c_{u,1}}{d(q-1)}}&*T^{\frac{(p-1)c_{u,1}}{d(q-1)}}&\cdots& *T^{\frac{(p-1)c_{u,1}}{d(q-1)}}&\cdots\\
\vdots & \vdots  & \ddots&\vdots&\\
*T^{\frac{(p-1)c_{u,n}}{d(q-1)}}&*T^{\frac{(p-1)c_{u,n}}{d(q-1)}}&\cdots& *T^{\frac{(p-1)c_{u,n}}{d(q-1)}}&\cdots\\
\vdots & \vdots & \ddots  & \vdots  & \ddots  
\end{pmatrix},
\end{equation}
where all $*$ in this matrix are from $\ZZ_q[\![T^{\frac{1}{d(q-1)}}]\!]$.

Running an analogous argument to \cite[Corollary~3.9]{ren-wan-xiao-yu}, we obtain 
\begin{equation}\label{dwork}
\det\Big(1-\psi^as\;\big|\;\widetilde{\textbf{B}}_u/\ZZ_q[\![T^{\frac{1}{d(q-1)}}]\!]\Big)= \det\big(I-s \sigma^{a-1}(N) \cdots \sigma(N) N \big).
\end{equation}

\begin{notation}
	For a matrix $M$, we write \[\left[ \begin{array}{cccccccccc}
	m_1 & m_2 &\cdots&m_{k} \\
	n_1 & n_2 &\cdots&n_{k}  \end{array} \right]_M\]
	for the $k\times k$-submatrix formed by elements whose row indices belong to $\{m_1,\dots,m_{k}\}$ and whose column indices belong to $\{n_1,\dots,n_{k}\}$.
\end{notation}

 \begin{lemma}\label{lower bound}
 	Let $(t_{ij})_{j\in \ZZ_{\geq 0}}$ be $n$ non-decreasing sequences, and let $M_1,M_2,\dots,M_n$ be $n$ nuclear matrices such that $$M_i=\Diag(T^{t_{i1}},T^{t_{i2}},\dots)\cdot M_i'\quad \textrm{for any}\ 1\leq i\leq n,$$
 	where $M_i'$ are infinite matrix whose entries belong to $\ZZ_q[\![T^{\frac{1}{d(q-1)}}]\!]$.
	 Then the $T$-adic Newton polygon $$\NP_T\big(\det(I-M_{n} \cdots M_2 M_1s)\big)\geq \NP\Big(\Big\{(k,\sum_{i=1}^n\sum\limits_{j=1}^{k}(t_{ij}))\;|\; k\geq 0\Big\}\Big).$$
 \end{lemma}
\begin{proof}
	
Put $$\det(I-sM_{n} \cdots M_2 M_1):=\sum\limits_{k=0}^\infty (-1)^k \mathbbm r_k(T) s^k .$$
		From the definition of characteristic power series, we get 
		\begin{equation}
		\label{E:expression of char power series}
		\begin{split}
		\mathbbm r_k(T)& =\sum\limits_{0\leq m_1<m_2<\cdots<m_{k}<\infty}\det \left[ \begin{array}{cccccccccc}
		m_1 & m_2 &\cdots&m_{k} \\
		m_1 & m_2 &\cdots&m_{k}  \end{array} \right]_{M_{n} \cdots M_2 M_1}     \\
		&=\sum_{\substack{0\leq m_{1,1}<m_{1,2}<\cdots<m_{1,k}<\infty \\ \cdots\\0\leq m_{n,1}<m_{n,2}<\cdots<m_{n,k}<\infty}} \det\bigg(\prod\limits_{i=1}^{n}\left[ \begin{array}{cccccccccc}
		m_{i+1,1} & m_{i+1,2} &\cdots&m_{i+1,k} \\
		m_{i,1} & m_{i,2} &\cdots&m_{i,k}  \end{array} \right]_{M_i}\bigg)            \\
		&=\sum_{\substack{0\leq m_{1,1}<m_{1,2}<\cdots<m_{1,k}<\infty \\ \cdots\\0\leq m_{n,1}<m_{n,2}<\cdots<m_{n,k}<\infty}} \prod\limits_{i=1}^{n}\bigg(
		\det \left[ \begin{array}{cccccccccc}
		m_{i+1,1} & m_{i+1,2} &\cdots&m_{i+1,k} \\
		m_{i,1} & m_{i,2} &\cdots&m_{i,k} \end{array} \right]_{M_i}\bigg).
		\end{split}
		\end{equation}
		Here and after, we set $m_{n+1,i}=m_{1,i}$ for all $1\leq i\leq k$.
		Since $$v_T\bigg(\det \left[ \begin{array}{cccccccccc}
			m_{i+1,1} & m_{i+1,2} &\cdots&m_{i+1,k} \\
			m_{i,1} & m_{i,2} &\cdots&m_{i,k} \end{array} \right]_{M_i}\bigg)\geq 
			\sum_{j=1}^{k}t_{ij},$$
			we complete the proof. 
\end{proof}

\begin{definition}\label{definition of HP}
	The \emph{Hodge polygon} of $C_f(\omega^u,T,s)$, denoted by $\HP(d,\omega^u,T)$, is the lower convex hull of set $$\Big\{\Big(k, \frac{a(p-1)}{db_u(q-1)}\sum\limits_{j=0}^{kb_u-1}c_{u,j}\Big)\;\big|\; k\geq 0\Big\}.$$
\end{definition}
\begin{lemma}
	Each point in $\Big\{(k, \frac{a(p-1)}{db_u(q-1)}\sum\limits_{j=0}^{kb_u-1}c_{u,j})\Big\}$ is a vertex of $\HP(d,\omega^u,T).$
\end{lemma}
\begin{proof}
	It follows that sequence $\Big(\frac{a(p-1)}{db_u(q-1)}\sum\limits_{j=(k-1)b_u}^{kb_u-1}c_{u,j}\Big)_{k\in \ZZ_{\geq 0}}$ is strictly increasing in $k$.
\end{proof}

Recall $$u=\sum\limits_{j=0}^{a-1}u(j)p^j\quad \textrm{and}\quad y_u(k)=\frac{ak(k-1)(p-1)}{2d}+\frac{k\sum\limits_{j=0}^{a-1}{u(j)}}{d}.$$
\begin{lemma}\label{computation}
	 We have $$y_u(k)=\frac{a(p-1)}{db_u(q-1)}\sum\limits_{j=1}^{kb_u}c_{u,j}.$$
\end{lemma}

\begin{proof}
	From \eqref{c}, we know
	\begin{multline}\label{sum of c}
	\sum\limits_{j=0}^{kb_u-1}c_{u,j}=\sum\limits_{j=0}^{k-1}\sum\limits_{\ell=0}^{b_u-1}c_{u,jb_u+\ell}=\sum\limits_{j=0}^{k-1}\Big[jb_u(q-1)+\sum_{i=0}^{b_u-1}u_i\Big]\\
	=\frac{k(k-1)b_u(q-1)}{2}+k\sum_{i=0}^{b_u-1}u_i.
	\end{multline}
	Since \begin{equation}\label{sum of u}\begin{split}
	\sum_{i=0}^{b_u-1}u_i=\frac{b_u}{a}(\sum\limits_{i=0}^{a-1}up^i\%(q-1))
	= \frac{b_u}{a}(\sum\limits_{j=0}^{a-1}u(j)\frac{q-1}{p-1}),
	\end{split}\end{equation}
	we know
	\[\begin{split}
	\frac{a(p-1)}{db_u(q-1)}\sum\limits_{j=1}^{kb_u}c_{u,j}=&\frac{a(p-1)}{db_u(q-1)}\Big(\frac{k(k-1)b_u(q-1)}{2}+\frac{k(q-1)b_u\sum\limits_{i=0}^{b_u-1}u_i}{a(p-1)}\Big)\\=&\frac{ak(k-1)(p-1)}{2d}+\frac{k\sum\limits_{j=0}^{a-1}u(j)}{d}=y_u(k).\qedhere
	\end{split}\]
\end{proof}

\begin{corollary}\label{passing some certain points}
The Hodge polygon $\HP(d,\omega^u,T)$ passes through the points $$\Big(k, y_u(k)\Big)\quad \textrm{for any}\ k\geq 0.$$ 
\end{corollary}

\begin{proposition}\label{HP}
The polygons $\NP_f(C,\omega^u,T)$ and $\HP(d,\omega^u,T)$ satisfy
	$$\NP_f(C,\omega^u,T)\geq \HP(d,\omega^u,T).$$
\end{proposition}
\begin{proof}

Since the matrix $N$ is nuclear as in \eqref{E:explicit N}, its conjugates $\sigma^i(N)$ are also nuclear matrices of the same form. Therefore, applying Lemma~\ref{lower bound} to the product of these matrices yields
$$\NP_T\Big(\det\big(I-s \sigma^{a-1}(N) \cdots \sigma(N) N \big)\Big)\geq \NP\Big(\Big\{\big(k, \frac{a(p-1)}{d(q-1)}\sum\limits_{j=1}^kc_{u,j}\big)\;\big|\;k\geq 0\Big\}\Big).$$

	From \eqref{ll} and \eqref{dwork}, we have 
	\[
	\NP_f(C,\omega^u,T)\geq \NP\Big(\Big\{\big(k, \frac{a(p-1)}{d(q-1)}\sum\limits_{j=1}^kc_{u,j}\big)\;\big|\;k\geq 0\Big\}\Big)=\HP(d,\omega^u,T).\qedhere
\]
\end{proof}
\begin{corollary}\label{HP for chi}
	For any character $\chi:\ZZ_p\to \CC_p^\times$ with conductor $p^{m_\chi}$, we have $$\NP_f(C,\omega^u,\chi)\geq \frac{1}{p^{m_\chi-1}(p-1)}\HP(d,\omega^u,T).$$
\end{corollary}
\begin{proof}
	It simply follows
	\[
	\NP_f(C,\omega^u,\chi)\geq \frac{1}{p^{m_\chi-1}(p-1)}\NP_f(C,\omega^u,T)\geq\frac{1}{p^{m_\chi-1}(p-1)}\HP(d,\omega^u,T).\qedhere
\]
\end{proof}

\section{Proof of Theorem~\ref{main} and Theorem~\ref{strong}}
In this section, we prove the main theorems.
\begin{proposition}\label{coincide for chi1}
	\hfill
\begin{itemize}
	\item[(a)] The Newton polygon $\NP_f(C,\omega^u,T)$ passes through the points $$\Big(kd,y_u(kd)\Big)\quad \textrm{for all}\ k\geq 0.$$
	\item[(b)] If we write  \begin{equation}\label{ru}
	C_f(\omega^u,T,s)=\sum\limits_{k=0}^\infty r_{u,k}(T)s^k,
	\end{equation} then for any $k\geq 0$ and $0\leq u\leq q-2$, the leading term of $r_{u,kd}$ is of the form $$*T^{y_u(kd)},$$ where $*$ represents a $p$-adic unit. 
\end{itemize}	
\end{proposition}
\begin{notation}\label{def of UP}
We denote by $\UP(d,\omega^u,T)$ the lower convex hull of the points in $$\Big\{(kd,y_u(kd))\;\big|\;k\geq 0\Big\}.$$
\end{notation}
\begin{corollary}\label{upper T}
	The polygon $\UP(d,\omega^u,T)$ forms an upper bound of $\NP_f(C,\omega^u,T)$.
\end{corollary}
\begin{proof}
	This follows directly from Proposition~\ref{coincide for chi1} (a).
\end{proof}
\begin{corollary}\label{upper chi}
	Any finite character $\chi:\ZZ_p\to \CC_p^\times$ with conductor $p^{m_\chi}$ satisfies \begin{equation}
	\NP_f(C,\omega^u,\chi)\leq\frac{1}{(p-1)p^{m_\chi-1}}\UP(d,\omega^u,T).
	\end{equation}
\end{corollary}
\begin{proof}
	It follows from Theorem~\ref{coincide for chi1}~(b).
\end{proof}
We will give the proof of Proposition~\ref{coincide for chi1} later.

\begin{lemma}\label{technical lemma}
	Let $\NP_1, \NP_2,\dots,\NP_n$ be $n$ Newton polygons. Assume for each $1\leq i\leq n$ there is a rational number $c$ and a vertex $(k_i, y_i)$ of $\NP_i$ such that all segments of $\NP_i$ before this point have slopes strictly less than $c$, while all segments after that point have slopes greater than $c$. Then $\bigoplus\limits_{i=1}^n \NP_i$ passes though the point $$\Big(\sum_{i=1}^{n}k_i, \sum_{i=1}^{n}y_i\Big).$$
	
\end{lemma}
\begin{proof}
	The proof follows from the definition of direct sum ``$\oplus$'' of polygons. 
\end{proof}
\begin{lemma}\label{specialization}
	Any finite character $\chi$ with conductor $p^{m_\chi}$ satisfies
	\begin{equation}\label{ineq}
	(p-1)p^{m_\chi}\NP_g(C,\omega^0,\chi)\geq \NP_g(C,\omega^0,T).
	\end{equation}
\end{lemma}
\begin{proof}
	It is enough to show each monomial $aT^i\in \ZZ_q[T]$ satisfy $$v_p(a(\chi(1)-1)^i)\geq v_T(aT^i),$$ which follows
\begin{multline*}
	(p-1)p^{m_\chi}v_p(a(\chi(1)-1)^i)=(p-1)p^{m_\chi}(v_p(a)+iv_p(\chi(1)-1))\\=
	(p-1)p^{m_\chi}v_p(a)+i\geq v_T(aT^i).\qedhere
\end{multline*}
\end{proof}
\begin{proof}[Proof of Proposition~\ref{coincide for chi1}]

	Proof of (a). Fix a finite character $\chi_1:\ZZ_p\to \CC_p^\times$ with conductor $p$. 
	By Lemma~\ref{specialization}, we have
	\begin{equation}
	 (p-1)\NP_g(C,\omega^0,\chi_1)\geq \NP_g(C,\omega^0,T).
	 \end{equation}
	 
	 By \cite[Proposition~3.2]{Davis-wan-xiao}, the $p$-adic Newton polygon $\NP_g(C,\omega^0,\chi_1)$ passes through the points $$\Big(kd(q-1), \frac{ak\big((q-1)kd-1\big)(p-1)}{2}\Big)\quad \textrm{for any}\ k\geq 0.$$
	 Hence, we know that $\NP_g(C,\omega^0,T)$ is not above point $$\Big(kd(q-1), \frac{ak((q-1)kd-1)(p-1)}{2}\Big)\quad \textrm{for any}\ k\geq 0. $$
	
	On the other hand, by Definition~\ref{definition of HP} and Lemma~\ref{computation}, we have 
	\begin{itemize}
		\item[(1)] For any $0\leq u\leq q-2$ and $k\geq 0$, the point $$\Big(kd,y_u(kd)\Big)$$ is a vertex of $\HP(d,\omega^u,T)$.
		
		\item[(2)] All segments of $\HP(d,\omega^u,T)$ before this point have slopes strictly less than $ak(p-1)$, while all segments after this point have slopes greater than $ak(p-1)$. 
	\end{itemize}
	
	By checking the conditions in Lemma~\ref{technical lemma}, we prove $\bigoplus\limits_{u=0}^{q-2}\HP(d,\omega^u,T)$ passes through the points $$\Big(kd(q-1), \sum\limits_{u=0}^{q-2}y_u(kd)\Big).$$

	Combining it with Proposition~\ref{HP} yields that $\NP_g(C,\omega^0,T)$ is not above the points $$\Big(kd(q-1), \sum\limits_{u=0}^{q-2}y_u(kd)\Big)\quad \textrm{for any}\ k\geq 0.$$
	 Thus, 
	\begin{equation}
	\label{fake inequality}\sum\limits_{u=0}^{q-2}\Big(\frac{ak(kd-1)(p-1)}{2}+k\sum\limits_{i=0}^{a-1}u(i)\Big)=\sum\limits_{u=0}^{q-2}y_u(kd)\leq \frac{ak((q-1)kd-1)(p-1)}{2}.
	\end{equation}
	Now we show that \eqref{fake inequality} is actually an equality. 
	
	Consider \[
	\sum\limits_{u=0}^{q-2}\sum\limits_{i=0}^{a-1}u(i)=-a(p-1)+\sum\limits_{i=0}^{a-1}\sum\limits_{u=0}^{q-1}u(i)=-a(p-1)+\sum\limits_{i=0}^{a-1}q\frac{p-1}{2}
	=\frac{aq(p-1)}{2}-a(p-1). 
\]
	Then we simplify the left-hand side of \eqref{fake inequality} by
	\[\begin{split}
		\sum\limits_{u=0}^{q-2}\Big(\frac{ak(kd-1)(p-1)}{2}+k\sum\limits_{i=0}^{a-1}u(i)\Big)=&(p-1)\Big(\frac{aqk}{2}-ak+\frac{(q-1)ak(kd-1)}{2}\Big)\\
	=&\frac{ak((q-1)kd-1)(p-1)}{2},
	\end{split}\]
	which is equal to its right-hand side.
	It implies for any $u\in \{0,1,\dots,q-2\}$, the Newton polygon $\NP_f(C,\omega^u,T)$ passes through the points  $$\Big(kd,y_u(kd)\Big)\quad \textrm{for any}\ k\geq 0.$$
	
	Proof of (b). From (a), we are able to write \begin{equation}\label{expression of ruk}
	r_{u,kd}(T):=\sum\limits_{i=y_u(kd)}^\infty r_{u,kd,i}T^i,
	\end{equation}
	where $r_{u,kd,i}$ belongs to $\calO_{\CC_p}$.
	
	Put $C_g(\omega^0,T,s)=\sum\limits_{n=0}^{\infty}w_n(T)s^{n}.$
	From \cite{Davis-wan-xiao}, we know that the leading term of $w_{kd}(T)$ has the form $$*_kT^{\frac{ak((q-1)kd-1)(p-1)}{2}},$$ where $*_k$ is a $p$-adic unit.
	It is easy to show that $$\prod\limits_{u=0}^{q-2} r_{u,kd,y(kd)}=*_k ,$$ 
	which implies that $r_{u,kd,i}$ are all $p$-adic units. 
\end{proof}
%	Therefore, by Lemma~\ref{sum inequality}, we obtain $\bigoplus_{u=0}^{q-2} HP(d,\omega^u,T)\leq \bigoplus_{u=0}^{q-2}\NP_f(C,\omega^u,T)$. Combining it with Lemma~\ref{core lemma}, we know that $$\bigoplus_{u=0}^{q-2} HP(d,\omega^u,T)\leq 	\NP_g(C,\omega^0,T).$$ 

Now we are ready to prove our main theorems of this paper.
\begin{proof}[Proof of Theorem~\ref{main}]
	(a) From \eqref{specialization of L}, we obtain
	\begin{equation*}
	C_f(\omega^u,T,s)|_{T=\chi(1)-1}=C_f(\omega^u,\chi,s)
	=\sum\limits_{k=0}^\infty r_{u,k}(\chi(1)-1)s^k.
	\end{equation*}
	Therefore, by Proposition~\ref{coincide for chi1} (b), the Newton polygon $(p-1)p^{m_\chi-1}\NP_f(C,\omega^u,\chi)$ is not above point  
	$$\Big(kd,y_u(kd)\Big)\quad \textrm{for all}\ k\geq 0.$$
	
	On the other hand, the Hodge polygon $\HP(d,\omega^u,T)$ forms a lower bound of $(p-1)p^{m_\chi-1}\NP_f(C,\omega^u,\chi)$ and for all $k\geq 0$ the points $$\Big(kd,y_u(kd)\Big)$$ are also vertices of $\HP(d,\omega^u,T)$. 
	
	Therefore, the points $$\Big(kd,y_u(kd)\Big)$$ are forced to be the vertices of $(p-1)p^{m_\chi-1}\NP_f(C,\omega^u,\chi)$. 
	
	A simple argument about the relation between roots of a power series and its $p$-adic Newton polygon completes the proof.
	
	(b) Since the slopes of segments of $\HP(d,\omega^u,T)$ between $x=d(k-1)$ and $x=dk$ are in the interval $$\Big[	a(k-1)(p-1)+\frac{1}{d}\sum\limits_{i=0}^{a-1}u(i), a(k-1)(p-1)+	\frac{a}{d}(d-1)(p-1)+\frac{1}{d}\sum\limits_{i=0}^{a-1}u(i)
	\Big],$$ by simply applying (a), we know that the slopes of segments of $(p-1)p^{m_\chi-1}\NP_f(C,\omega^u,\chi)$ between $x=d(k-1)$ and $x=dk$ also in this interval, which completes the proof of (b).
\end{proof}
Recall $\UP(d,\omega^u,T)$ is the upper bound of $\NP_f(C,\omega^u,T)$ defined in Notation~\ref{def of UP}.
\begin{lemma}\label{distance}
	The vertical distance between points in $\UP(d,\omega^u,T)$ and $\NP_f(C,\omega^u,T)$ is bounded above by $\frac{ad(p-1)}{8}$. 
\end{lemma}
\begin{proof}
	By Corollary~\ref{upper T} and Proposition~\ref{HP}, we know $$\UP(d,\omega^u,T)\geq  \NP_f(C,\omega^u,T)\geq \HP(d,\omega^u,T).$$
	By Corollary~\ref{passing some certain points}, the polygon  $\HP(d,\omega^u,T)$ is above the parabola $P$ defined by $$P(x):= \frac{ax(x-1)(p-1)}{2d}+\frac{x\sum\limits_{i=0}^{a-1}u(i)}{d}.$$
	Since all vertices  $(kd,y_u(kd))$ of $\UP(d,\omega^u,T)$ coincide with the parabola $P$, by simple calculation, the maximal vertical distance of $\UP(d,\omega^u,T)$ and $P$ is equal to 
	\[
	\max\limits_{k\in \ZZ_{\geq 0}}\Big\{\frac{P(d(k+1))+P(d(k))}{2}-P(d(k+1/2))\Big\}
	=\max\limits_{k\in \ZZ_{\geq 0}} \Big\{\frac{ad(p-1)}{8}\Big\}=\frac{ad(p-1)}{8}.\qedhere
	\]
\end{proof}
\begin{proposition}\label{independent}
	Let $\chi$ be a finite character with conductor $p^{m_\chi}> \frac{adp}{8}$. Then the Newton polygon $p^{m_\chi}\NP_f(C,\omega^u,\chi)$ is independent of $\chi$. 
\end{proposition}
\begin{proof}
	Recall in \eqref{ru} we denote $$C_f(\omega^u,T,s)=\sum\limits_{k=0}^\infty r_{u,k}(T)s^k.$$  By Proposition~\ref{HP} and Corollary~\ref{passing some certain points}, we are able to write $r_{u,k}(T)$ of the form $$r_{u,k}(T)=\sum\limits_{j=y_u(k)}^\infty r_{u,k,i}T^j.$$
	Assume that $i(k)$ is the smallest integer such that 
	\begin{itemize}
		\item $i(k)\leq \UP(d,\omega^u,T)(k)$, where $\UP(d,\omega^u,T)(k)$ is the height of $\UP(d,\omega^u,T)$ at $x=k$.
		\item The corresponding coefficient $r_{u,k,i(k)}$ is a $p$-adic unit.
	\end{itemize}
	If such $i(k)$ does not exist, we simply put $i(k)=\infty$.
	
	Then we will show that for any $\chi$ satisfying \begin{equation}\label{condition for chi}
	p^{m_\chi}> \frac{adp}{8},
	\end{equation} 
	the Newton polygon $p^{m_\chi-1}(p-1)\NP_f(C,\omega^u,\chi)$ is the same as  $\NP\Big(\Big\{(k,i(k))\;\big|\;k\geq 0\Big\}\Big).$
	
	Since  $C_f(\omega^u,T,s)\in \ZZ_p[\![T]\!],$ for any $\ell< i(k)$ we have \begin{equation}\label{ineq1}
	v_p\big(r_{u,k,\ell}(\chi(1)-1)^\ell\big)\geq 1+(\frac{\ell}{p^{m_\chi-1}(p-1)})
	=\frac{p^{m_\chi-1}(p-1)+\ell-i(k)+i(k)}{p^{m_\chi-1}(p-1)}.
	\end{equation}
	By Lemma~\ref{distance} and the definition of $i(k)$, we know that \begin{equation}\label{aaa}
	i(k)-\ell\leq \frac{ad(p-1)}{8}.
	\end{equation}
	It follows from the inequalities \eqref{aaa}, \eqref{condition for chi} and \eqref{ineq1} that
\begin{multline*}
	v_p\Big(r_{u,k,\ell}\cdot(\chi(1)-1)^\ell\Big)\geq \frac{p^{m_\chi-1}(p-1)-\frac{ad(p-1)}{8}+i(k)}{p^{m_\chi-1}(p-1)}\\
	> \frac{i(k)}{{p^{m_\chi-1}(p-1)}}
	=v_p\Big(r_{u,k,i(k)}\cdot(\chi(1)-1)^{i(k)}\Big).
\end{multline*}
	The inequality above implies that 
	$v_p\big(u_{u,k}\cdot(\chi(1)-1)\big)$ is  
	\begin{itemize}
		\item either equal to $\frac{i(k)}{{p^{m_\chi-1}(p-1)}}$
		\item or greater than $\frac{1}{(p-1)p^{m_\chi-1}}\UP(d,\omega^u,T)(k)$.
	\end{itemize}
	Then this proposition follows directly from Corollary~\ref{upper chi}.
\end{proof}
For a Newton polygon $\NP$ and a rational number $t$ recall the definition of Newton polygon $t+\NP$ in Notation~\ref{properties for NP}~(d). 
\begin{lemma}
	Let $\chi:\ZZ_p\to \CC_p^\times$ be a finite character with conductor $p^{m_\chi}$. Then we have
	\begin{equation}
	\Big\{\alpha\in R(\NP_f(C,\omega^u,\chi))\;\big|\;\alpha< ak\Big\}=\biguplus_{i=0}^{k-1}R\Big(ai+\NP_f(L,\omega^u,\chi)\Big),
	\end{equation}
	where $R$ is defined in Definition~\ref{definition of R}.
\end{lemma}
\begin{proof}
	Since $$C_f(\omega^u,\chi,s)=\prod_{i=0}^{\infty}L_f(\omega^u,\chi,q^is)\quad\textrm{and}\quad R\big(\NP_f(L,\omega^u,\chi)\big)\subset [0,a),$$ 
	we know that 
\begin{multline*}
	\Big\{\alpha\in R(\NP_f(C,\omega^u,\chi))\;\big|\;\alpha< ak\Big\}=\biguplus_{i=0}^{k-1} \Big\{p\textrm{-adic Newton slopes of}\ L_f(\omega^u,\chi,q^{i}s)\Big\}\\=\biguplus_{i=0}^{k-1}R\Big(ai+\NP_f(L,\omega^u,\chi)\Big).\qedhere
\end{multline*}

\end{proof}
\begin{proof}[Proof of Theorem~\ref{strong}]
	By Proposition~\ref{independent}, we have 
	\[\begin{split}
	&R\big(\NP_f(L,\omega^u,\chi)\big)\\
	=&\Big\{\alpha\in R\big(\NP_f(C,\omega^u,\chi)\big)\;\Big|\;\alpha< a\Big\}\\
	=&\Big\{\frac{\alpha}{p^{m_\chi}}\;\Big|\; \alpha\in p^{m_\chi}R\big(\NP_f(C,\omega^u,\chi)\big)\ \textrm{and}\ \alpha< ap^{m_\chi}\Big\}\\
	=&\Big\{\frac{\alpha}{p^{m_\chi}}\;\Big|\; \alpha\in p^{m_0}R\big(\NP_f(C,\omega^u,\chi_0)\big)\ \textrm{and}\ \alpha< ap^{m_\chi}\Big\}\\
	=&\Big\{\frac{\alpha p^{m_0}}{p^{m_\chi}}\;\Big|\; \alpha\in R\big(\NP_f(C,\omega^u,\chi_0)\big)\ \textrm{and}\ \alpha< ap^{m_\chi-m_0}\Big\}\\
	=&\biguplus_{i=0}^{p^{m_\chi-m_0}-1}R\Big(\frac{1}{p^{m_\chi-m_0}}\big(ai+\NP_f(L,\omega^u,\chi_0)\big)\Big)\\
	=&\bigcup_{i=0}^{p^{m_\chi-m_0}-1}\Big\{\frac{\alpha_1+ai}{p^{m_\chi-m_0}},\frac{\alpha_2+ai}{p^{m_\chi-m_0}},\dots, \frac{\alpha_{dp^{m_0-1}}+ai}{p^{m_\chi-m_0}}\Big\}.\qedhere
	\end{split}\]
\end{proof}


\begin{thebibliography}{9999}

\bibitem[BFZ]{bfz}
R. Blache, E. Ferard, and H. Zhu,
Hodge--Stickelberger polygons for $L$-functions of exponential sums of $P(x^s)$, 
{\it Math. Res. Lett.}  {\bf 15} (2008), no. 5, 1053--1071.

\bibitem[DWX]{Davis-wan-xiao}
C. Davis, D. Wan and L. Xiao,
Newton slopes for Artin--Schreier--Witt towers, {\it Math. Ann.}  {\bf 364} (2016), no. 3, 1451--1468. 



\bibitem[Li]{li}
X. Li, The stable property of Newton slopes for general
Witt towers, {\it J. Number Theory.}
{\bf 185},  (2018), 144--159.
 

\bibitem[LW]{liu-wan}
C. Liu and D. Wan,
$T$-adic exponential sums over finite fields, {\it Algebra and Number Theory} {\bf3} (2009), no. 5, 489--509. 


%\bibitem[LWei]{liu-wei} C. Liu and D. Wei, 
%The $L$-functions of Witt coverings, {\it Math. Z.}  {\bf 255} (2007), 95--115. 

\bibitem[LL]{liu-liu}
C. Liu and W. Liu,
Twisted exponential sums of polynomials in one variable, {\it Science China(Mathematics)} {\bf 53}
(2010), no. 9, 2395--2404.

\bibitem[Liu]{liu-liu-niu} 
C. Liu, W. Liu, C. Niu,
$T$-adic exponential sums under diagonal base change, {\it J. Number Theory} {\bf 166} (2016), 276--297.


\bibitem[Liu-Wei]{liu-wei} C. Liu and D. Wei, 
The $L$-functions of Witt coverings, {\it Math. Z.}  {\bf 255} (2007), 95--115. 

 \bibitem[RWXY]{ren-wan-xiao-yu}
 R. Ren, D. Wan, L. Xiao, and M. Yu,
 Slopes for higher rank Artin--Schreier--Witt Towers, Trans. Amer. Math. Soc. {\bf 370} (2018), 6411--6432.



\end{thebibliography}
\end{document}